\documentclass[10pt]{amsart}
\usepackage{amssymb}
\usepackage{bm}
\usepackage{graphicx}
\usepackage[centertags]{amsmath}
\usepackage{amsfonts}
\usepackage{amsthm}
\newtheorem{thm}{Theorem}

\newtheorem{lem}[thm]{Lemma}

\newtheorem{claim}[thm]{Claim}
\newtheorem{fact}[thm]{Fact}
\newtheorem{problem}{Problem}
\newtheorem{defn}[thm]{Definition}
\theoremstyle{definition}

\newcommand{\nn}{\mathbb{N}}
\newcommand{\ee}{\varepsilon}

\newcommand{\aaa}{\mathcal{A}}
\newcommand{\bbb}{\mathcal{B}}
\newcommand{\ccc}{\mathcal{C}}
\newcommand{\ddd}{\mathcal{D}}

\newcommand{\fff}{\mathcal{F}}
\newcommand{\hhh}{\mathcal{H}}
\newcommand{\iii}{\mathcal{I}}
\newcommand{\kkk}{\mathcal{K}}
\newcommand{\rrr}{\mathcal{R}}

\newcommand{\www}{\mathcal{W}}
\newcommand{\xxx}{\mathcal{X}}
\newcommand{\yyy}{\mathcal{Y}}

\newcommand{\ct}{2^{<\nn}}
\newcommand{\con}{\smallfrown}
\newcommand{\sg}{\sigma}
\newcommand{\incr}{\mathrm{Incr}}
\newcommand{\decr}{\mathrm{Decr}}
\newcommand{\seg}{\mathfrak{s}}

\begin{document}
\title{Operators whose dual has non-separable range}
\author{Pandelis Dodos}

\address{Department of Mathematics, University of Athens, Panepistimiopolis 157 84, Athens, Greece.}
\email{pdodos@math.uoa.gr}

\thanks{Research supported by NSF grant DMS-0903558.}

\maketitle


\begin{abstract}
Let $X$ and $Y$ be separable Banach spaces and $T:X\to Y$ be a bounded linear
operator. We characterize the non-separability of $T^*(Y^*)$ by means of
fixing properties of the operator $T$.
\end{abstract}


\section{Introduction}

The study of fixing properties of certain classes of
operators\footnote[1]{Throughout the paper by the term \textit{operator}
we mean bounded, linear operator.} between separable Banach spaces is a
heavily investigated part of Banach Space Theory which is closely
related to some central questions, most notably with the problem
of classifying, up to isomorphism, all complemented subspaces of classical
function spaces (see \cite{Ro5} for an excellent exposition).

Typically, one has an operator $T:X\to Y$ which is ``large" in a suitable
sense and tries to find a concrete object that the operator $T$ preserves.
Various versions of this problem have been studied in the literature
and several satisfactory answers have been obtained; see, for instance,
\cite{Al,Bou1,Bou2,FGJ,FG,Ga1,Ga2,Pe,Ro1}. Among them, there are
two fundamental results that deserve special attention. The first one is
due to A. Pe{\l}czy\'{n}ski and asserts that every non-weakly compact
operator $T:C[0,1]\to Y$ must fix a copy\footnote[2]{An operator
$T:X\to Y$ is said to \textit{fix a copy} of a Banach space $E$
if there exists a subspace $Z$ of $X$ which is isomorphic to $E$
and is such that $T|_Z$ is an isomorphic embedding.} of $c_0$.
The second result is due to H. P. Rosenthal and asserts that every
operator $T:C[0,1]\to Y$ whose dual $T^*$ has non-separable range
must fix a copy of $C[0,1]$.

The present paper is a continuation of this line of research and is devoted
to the study of the following problem.
\begin{problem} \label{pr1}
Let $X$ and $Y$ be separable Banach spaces and $T:X\to Y$ be an operator
such that $T^*$ has non-separable range. What kind of fixing properties does
the operator $T$ have?
\end{problem}

To state our main results we need to fix some pieces of notation
and introduce some terminology. By $\ct$ we shall denote the
Cantor tree. By $\varphi:\ct\to\nn$ we denote the unique bijection
satisfying $\varphi(s)<\varphi(t)$ if either $|s|<|t|$ or $|s|=|t|=n$ and
$s<_{\mathrm{lex}} t$ (here $<_{\mathrm{lex}}$ stands for the usual
lexicographical order on $2^n$). We recall the following class of basic
sequences (see \cite{ADK1,ADK2,D2}).
\begin{defn} \label{d11}
Let $X$ be a Banach space and $(x_t)_{t\in\ct}$ be a sequence in $X$ indexed
by the Cantor tree. We say that $(x_t)_{t\in\ct}$ is \emph{topologically
equivalent to the basis of James tree} if the following are satisfied.
\begin{enumerate}
\item[(1)] If $(t_n)$ is the enumeration of $\ct$ according to the bijection
$\varphi$, then the sequence $(x_{t_n})$ is a seminormalized basic sequence.
\item[(2)] For every infinite antichain $A$ of $\ct$ the sequence
$(x_t)_{t\in A}$ is weakly null.
\item[(3)] For every $\sg\in 2^\nn$ the sequence $(x_{\sg|n})$ is weak* convergent
to an element $x^{**}_\sg\in X^{**}\setminus X$. Moreover, if $\sg, \tau\in 2^\nn$
with $\sg\neq \tau$, then $x^{**}_\sg\neq x^{**}_\tau$.
\end{enumerate}
\end{defn}
The archetypical example of such a sequence is the standard unit vector basis
of James tree space $JT$ (see \cite{Ja}). There are also classical Banach spaces
having a natural Schauder basis topologically equivalent to the basis of
James tree; the space $C[0,1]$ is an example.

We now introduce the following definition.
\begin{defn} \label{d12}
Let $X$ and $Y$ be Banach spaces and $T:X\to Y$ be an operator. We say that
$T$ \emph{fixes a copy of a sequence topologically equivalent
to the basis of James tree} if the there exists a sequence $(x_t)_{t\in\ct}$
in $X$ such that both $(x_t)_{t\in\ct}$ and $(T(x_t))_{t\in\ct}$ are topologically
equivalent to the basis of James tree.
\end{defn}
We notice that if $T:X\to Y$ fixes a copy of a sequence $(x_t)_{t\in\ct}$
as above, then the topological structure of the weak* closure of
$\{x_t:t\in\ct\}$ in $X^{**}$ is preserved under the action of the
operator $T^{**}$ (see Lemma \ref{l530}). We point out, however,
that metric properties are not necessarily preserved (see \S 5.3).

We are ready to state the first main result of the paper.
\begin{thm} \label{t13}
Let $X$ be a separable Banach space not containing a copy of $\ell_1$,
$Y$ be a separable Banach space and $T:X\to Y$ be an operator. Then
the following are equivalent.
\begin{enumerate}
\item[(i)] The dual operator $T^*:Y^*\to X^*$ of $T$ has non-separable range.
\item[(ii)] The operator $T$ fixes a copy of a sequence topologically
equivalent to the basis of James tree.
\end{enumerate}
\end{thm}
The assumption in Theorem \ref{t13} that the space $X$ does not contain a
copy of $\ell_1$ is not redundant. Indeed, if $Q:\ell_1\to JT$ is a quotient map,
then the dual operator $Q^*$ of $Q$ has non-separable range yet $Q$ is strictly
singular\footnote[3]{Actually, every operator $T:\ell_1\to JT$ is strictly
singular since every infinite-dimensional subspace of $JT$ contains a copy
of $\ell_2$ (see \cite{Ja}).} and fixes no copy of a sequence topologically
equivalent to the basis of James tree. Observe, however, that in this case there
exists a bounded sequence $(x_t)_{t\in\ct}$ in $\ell_1$ such that its image
$(Q(x_t))_{t\in\ct}$ is topologically equivalent to the basis of James tree.
On the other hand, if $Q:\ell_1\to C[0,1]$ is a quotient map, then $Q$
fixes a copy of $\ell_1$. Our second main result shows that this phenomenon
holds true in full generality.
\begin{thm} \label{t14}
Let $X$ be a separable Banach space containing a copy of $\ell_1$,
$Y$ be a separable Banach space and $T:X\to Y$ be an operator. Then
the following are equivalent.
\begin{enumerate}
\item[(i)] The dual operator $T^*:Y^*\to X^*$ of $T$ has non-separable range.
\item[(ii)] Either the operator $T$ fixes a copy of $\ell_1$ or there exists a bounded
sequence $(x_t)_{t\in\ct}$ in $X$ such that its image $(T(x_t))_{t\in\ct}$ is topologically
equivalent to the basis of James tree.
\end{enumerate}
\end{thm}
The paper is organized as follows. In \S 2 we gather some background
material. In \S 3 we give the proof of Theorem \ref{t13} while
in \S 4 we give the proof of Theorem \ref{t14}. Finally, in \S 5
we make some comments.


\section{Background material}

Our general notation and terminology is standard as can be found, for instance,
in \cite{Kechris} and \cite{LT}. For every Banach space $X$ by $B_X$ we denote
the closed unit ball of $X$. By $\nn=\{0,1,2,...\}$ we shall denote the
natural numbers. If $S$ is a countable infinite set, then by $[S]^{\infty}$
we shall denote the set of all infinite subsets of $S$. Notice that $[S]^{\infty}$
is a $G_\delta$, hence Polish, subspace of $2^S$.

We will frequently need to compute the descriptive set-theoretic complexity
of various sets and relations. To this end, we will use the ``Kuratowski-Tarski
algorithm". We will assume that the reader is familiar with this classical method.
For more details we refer to \cite[page 353]{Kechris}.

\subsection{Subtrees of the Cantor tree}

As we have already mentioned, by $\ct$ we shall denote the Cantor tree;
i.e. $\ct$ is the set of all finite sequences of $0$'s and $1$'s
(the empty sequence is denoted by $\varnothing$ and is included in $\ct$).
We view $\ct$ as a tree equipped with the (strict) partial order $\sqsubset$
of extension. The \textit{length} of a node $t\in\ct$ is defined to be
the cardinality of the set $\{s\in\ct:s \sqsubset t\}$ and is denoted
by $|t|$. Two nodes $s,t\in\ct$ are said to be \textit{comparable} if either
$s\sqsubseteq t$ or $t\sqsubseteq s$. Otherwise, $s$ and $t$ are said
to be \textit{incomparable}. A subset of $\ct$ consisting of pairwise comparable
nodes is said to be a \textit{chain}, while a subset of $\ct$ consisting
of pairwise incomparable nodes is said to be an \textit{antichain}.
For every $s,t\in\ct$ we let $s\wedge t$ be the $\sqsubset$-maximal node $w$
of $\ct$ with $w\sqsubseteq s$ and $w\sqsubseteq t$. If $s,t\in\ct$ are
incomparable with respect to $\sqsubseteq$, then we write $s\prec t$ provided
that $(s\wedge t)^{\con}0 \sqsubseteq s$ and $(s\wedge t)^{\con}1\sqsubseteq t$.
For every $\sg\in 2^\nn$ and every $n\in\nn$ with $n\geq 1$ we set
$\sg|n=(\sg(0),...,\sg(n-1))$ while $\sg|0=\varnothing$.

\subsubsection{Downwards closed subtrees}

A non-empty subset $R$ of $\ct$ is said to be a \textit{downwards closed subtree}
if for every $t\in R$ and every $s\in\ct$ with $s\sqsubseteq t$ we have that
$s\in R$. The \textit{body} of a downwards closed subtree $R$ of $\ct$ is defined
to be the set $\{\sg\in 2^\nn:\sg|n\in R \ \forall n\in\nn\}$ and is denoted
by $[R]$. If $A$ is a non-empty subset of $\ct$, then the \textit{downwards closure}
of $A$ is defined to be the set $\{s\in \ct: \exists t\in A \text{ with }
s\sqsubseteq t\}$ and is denoted by $\hat{A}$; notice that $\hat{A}$
is a downwards closed subtree.

\subsubsection{Dyadic subtrees}

A subset $D$ of $\ct$ is said to be a \textit{dyadic subtree} if $D$
can be written in the form $\{s_t:t\in\ct\}$ so that for every $t_0, t_1\in\ct$
we have $t_0\sqsubset t_1$ (respectively $t_0\prec t_1$) if and only if
$s_{t_0} \sqsubset s_{t_1}$ (respectively $s_{t_0}\prec s_{t_1}$).
It is easy to see that such a representation of $D$ as $\{s_t: t\in\ct\}$ is unique.
In the sequel when we write $D=\{s_t:t\in\ct\}$, where $D$ is a dyadic subtree,
we will assume that this is the canonical representation of $D$ described above.
The notion of length and the binary relation $\wedge$ can be relativized
to any dyadic subtree $D$. In particular, if $s\in D$, then we let $|s|_D$
be the cardinality of the set $\{s'\in D: s'\sqsubset s\}$; moreover,
for every $s,s'\in D$ we let $s\wedge_D s'$ be the $\sqsubset$-maximal node
$w\in D$ such that $w\sqsubseteq s$ and $w\sqsubseteq s'$. Notice that if
$\{s_t:t\in\ct\}$ is the canonical representation of $D$, then for every
$t,t'\in\ct$ we have $|s_t|_D=|t|$ and $s_t\wedge_D s_{t'}=s_{t\wedge t'}$.

\subsubsection{Regular dyadic subtrees}

A dyadic subtree $D=\{s_t:t\in\ct\}$ is said to be \textit{regular}
if for every $t_0, t_1\in\ct$ we have $|t_0|=|t_1|$ if and only if
$|s_{t_0}|=|s_{t_1}|$; equivalently, the dyadic subtree $D=\{s_t:t\in\ct\}$ is regular
if for every $n\in\nn$ there exists $m\in\nn$ such that $\{s_t:t\in 2^n\}\subseteq 2^m$.

\subsection{Families of infinite sets and related combinatorics}

Throughout this subsection $S$ will be a countable infinite set. A family
$\aaa\subseteq [S]^{\infty}$ is said to be \textit{hereditary} if for
every $A\in\aaa$ and every $A'\in [A]^{\infty}$ we have that $A'\in \aaa$.

Given $A, B\in [S]^{\infty}$ we write $A\subseteq^* B$ if the set
$A\setminus B$ is finite, while we write $A\perp B$ if the set
$A\cap B$ is finite. Two families $\aaa, \bbb\subseteq [S]^{\infty}$
are said to be \textit{orthogonal} if $A\perp B$ for every $A\in\aaa$
and every $B\in\bbb$. A family $\aaa$ is said to be \textit{countably generated}
in a family $\bbb$ if there exists a sequence $(B_n)$ in $\bbb$ such that
for every $A\in\aaa$ there exists $n\in\nn$ with $A\subseteq^* B_n$.
A subfamily $\bbb$ of a family $\aaa$ is said to be \textit{cofinal} in
$\aaa$ if for every $A\in\aaa$ there exists $B\in [A]^{\infty}$ with $B\in\bbb$.

For every $\aaa\subseteq [S]^{\infty}$ we set
\begin{equation} \label{e1}
\aaa^\perp=\{ B\in [S]^{\infty}: B\perp A \text{ for every } A\in\aaa\}.
\end{equation}
The family $\aaa^\perp$ is called the \textit{orthogonal} of $\aaa$.
Clearly $\aaa^\perp$ is hereditary. Moreover, it is invariant
under finite changes; that is, if $B\in\aaa^\perp$ and $C\in [S]^{\infty}$
are such that $B\bigtriangleup C$ is finite, then $C\in\aaa^\perp$.

We recall the following class of hereditary families introduced in
\cite{DK}.
\begin{defn} \label{d25}
We say that a hereditary family $\aaa$ of infinite subsets of $S$ is an
\emph{M-family} if for every sequence $(A_n)$ in $\aaa$ there exists
$A\in\aaa$ whose all but finitely many elements are in $\bigcup_{n\geq k} A_n$
for every $k\in\nn$.
\end{defn}
The notion of an M-family is the ``hereditary" analogue of the notion of
a \textit{happy family} (also known as \textit{selective co-ideal}) introduced
by A. R. D. Mathias \cite{Ma}. We isolate, for future use, the following easy
fact (see \cite[Fact 3]{DK}).
\begin{fact} \label{f26}
Let $\aaa\subseteq [S]^{\infty}$ be a hereditary family. Then the following
are equivalent.
\begin{enumerate}
\item[(i)] The family $\aaa$ is an M-family.
\item[(ii)] For every sequence $(A_n)$ in $\aaa$ there exists $A\in\aaa$
such that $A\cap A_n\neq \varnothing$ for infinitely many $n\in\nn$.
\end{enumerate}
\end{fact}
Much of our interest on M-families stems from the fact that they possess
strong structural properties. To state the particular property we need,
we recall the following notion.
\begin{defn} \label{d27}
Let $\aaa, \bbb\subseteq [S]^{\infty}$ be two hereditary and orthogonal
families. A \emph{perfect Lusin gap} inside $(\aaa,\bbb)$ is a continuous,
one-to-one map $2^\nn\ni\sg \mapsto (A_\sg, B_\sg)\in \aaa\times \bbb$
such that the following are satisfied.
\begin{enumerate}
\item[(1)] For every $\sg\in 2^\nn$ we have $A_\sg\cap B_\sg=\varnothing$.
\item[(2)] For every $\sg, \tau\in 2^\nn$ with $\sg\neq \tau$ we have
$(A_\sg\cap B_\tau)\cup (A_\tau\cap B_\sg)\neq \varnothing$.
\end{enumerate}
\end{defn}
The notion of a perfect Lusin gap is due to S. Todorcevic \cite{To}
though it can be traced on earlier work of K. Kunen.

It is relatively easy to see that if $\aaa,\bbb\subseteq [S]^{\infty}$
are hereditary and orthogonal families and there exists
a perfect Lusin gap inside $(\aaa,\bbb)$, then $\aaa$ is \textit{not}
countably generated in $\bbb^\perp$. We will need the following theorem
which establishes the converse for certain pairs of orthogonal families
(see \cite[Theorem I]{DK}).
\begin{thm} \label{t28}
Let $\aaa, \bbb\subseteq [S]^{\infty}$ be two hereditary and orthogonal
families. Assume that $\aaa$ is analytic\footnote[4]{A subset $A$ of a
Polish space $X$ is said to be \textit{analytic} if there exists a Borel
map $f:\nn^\nn\to X$ such that $f(\nn^\nn)=A$. The complement of an
analytic set is said to be \textit{co-analytic}.} and that $\bbb$ is
an M-family and $C$-measurable\footnote[5]{A subset of a Polish space
is said to be \textit{$C$-measurable} if it belongs to the smallest
$\sigma$-algebra that contains the open sets and is closed under the
Souslin operation. All analytic and co-analytic sets are
$C$-measurable (see \cite{Kechris}).}. Then, either
\begin{enumerate}
\item[(i)] $\aaa$ is countably generated in $\bbb^\perp$, or
\item[(ii)] there exists a perfect Lusin gap inside $(\aaa,\bbb)$.
\end{enumerate}
\end{thm}

\subsection{Increasing and decreasing antichains of a regular dyadic tree}

We recall the following classes of antichains of the Cantor tree
introduced in \cite[\S 3]{ADK1}.
\begin{defn} \label{d29}
Let $D$ be a regular dyadic subtree of the Cantor tree.
An infinite antichain $(s_n)$ of $D$ will be called \emph{increasing}
if the following conditions are satisfied.
\begin{enumerate}
\item[(1)] For every $n,m\in\nn$ with $n<m$ we have $|s_n|_D<|s_m|_D$.
\item[(2)] For every $n,m,l\in\nn$ with $n<m<l$ we have
$|s_n|_D\leq |s_m\wedge_D s_l|_D$.
\item[(3I)] For every $n,m\in\nn$ with $n<m$ we have $s_n\prec s_m$.
\end{enumerate}
The set of all increasing antichains of $D$ will be denoted by $\incr(D)$.
Respectively, an infinite antichain $(s_n)$ of $D$ will be called
\emph{decreasing} if \emph{(1)} and \emph{(2)} above are satisfied and condition
\emph{(3I)} is replaced by the following.
\begin{enumerate}
\item[(3D)] For every $n,m\in\nn$ with $n<m$ we have $s_m\prec s_n$.
\end{enumerate}
The set of all decreasing antichains of $D$ will be denoted by
$\decr(D)$.
\end{defn}
We will need the following stability properties of the above defined
classes of antichains (see \cite[Lemma 8]{ADK1}).
\begin{lem} \label{l210}
Let $D$ be a regular dyadic subtree of $\ct$. Then the following hold.
\begin{enumerate}
\item[(i)] Let $(s_n)$ be an infinite antichain of $D$ and $(s_{n_k})$
be a subsequence of $(s_n)$. If $(s_n)\in \incr(D)$, then
$(s_{n_k})\in\incr(D)$. Respectively, if $(s_n)\in\decr(D)$,
then $(s_{n_k})\in\decr(D)$.
\item[(ii)] For every infinite antichain $(s_n)$ of $D$ there
exists a subsequence $(s_{n_k})$ of $(s_n)$ such that
either $(s_{n_k})\in\incr(D)$ or $(s_{n_k})\in\decr(D)$.
\item[(iii)] We have $\incr(D)=\incr(\ct)\cap D^\nn$
and $\decr(D)=\decr(2^\nn)\cap D^\nn$. In particular, if
$R$ is a regular dyadic subtree of $\ct$ with $R\subseteq D$,
then $\incr(R)=\incr(D)\cap R^\nn$ and $\decr(R)=\decr(D)\cap R^\nn$.
\end{enumerate}
\end{lem}
Notice that for every regular dyadic subtree $D$ of $\ct$
the sets $\incr(D)$ and $\decr(D)$ are Polish subspaces
of $D^\nn$. We will also need the following partition result
(see \cite[Theorem 10]{ADK1}).
\begin{thm} \label{t211}
Let $D$ be a regular dyadic subtree of $\ct$ and $\ccc$ be an
analytic subset of $D^\nn$. Then there exists a regular
dyadic subtree $R$ of $\ct$ with $R\subseteq D$ and such that
either $\incr(R)\subseteq\ccc$ or $\incr(R)\cap\ccc=\varnothing$.
Respectively, there exists a regular dyadic subtree $R'$ of $\ct$
with $R'\subseteq D$ and such that either $\decr(R')\subseteq\ccc$
or $\decr(R')\cap\ccc=\varnothing$.
\end{thm}
We notice that Theorem \ref{t211} is essentially a consequence
of the work of V. Kanellopoulos \cite{Ka} on Ramsey families of
subtrees of the Cantor tree.

\subsection{Selection of subsequences}

Let $X$ be a separable Banach space and for every $n\in\nn$
let $(x^n_k)$ be a weakly null sequence in $X$. If the dual
$X^*$ of $X$ is separable, then there exists a strictly increasing
sequence $(k_n)$ in $\nn$ such that the sequence $(x^n_{k_n})$
is also weakly null. This property fails if $X$ does not
contain a copy of $\ell_1$ and $X^*$ is non-separable
(see \cite{ADK1,ADK2}). Nevertheless, in this case
we have the following ``weak subsequence selection" principle
discovered by H. P. Rosenthal (see \cite[Theorem 3.6]{Ro4}).
\begin{thm} \label{t212}
Let $X$ be a separable Banach space and $\kkk$ be a weak* compact
subset of $X^{**}$ consisting only of Baire-1 functions\footnote[6]{An
element $x^{**}$ of $X^{**}$ is said to be \textit{Baire-1}
if $x^{**}$ is a Baire-1 function on $X^*$ when $X^*$ is
equipped with the weak* topology.}. For every
$n\in\nn$ let $(x^n_k)$ be a sequence in $\kkk$ which is
weak* convergent to an element $x^{**}_n$. Assume that
the sequence $(x^{**}_n)$ is weak* convergent to an element
$x^{**}$. Then there exists a sequence $(n_i,k_i)$ in $\nn\times\nn$
with $n_i< k_i<n_{i+1}$ for every $i\in\nn$ and such that the sequence
$(x^{n_i}_{k_i})$ is weak* convergent to $x^{**}$.
\end{thm}


\section{Proof of Theorem \ref{t13}}

This section is devoted to the proof of Theorem \ref{t13}. The fact
that (2) implies (1) follows from the following general fact.
\begin{lem} \label{l313}
Let $X$ and $Y$ be Banach spaces and $T:X\to Y$ be an operator. Assume that
there exists a bounded sequence $(x_t)_{t\in\ct}$ in $X$ such that its image
$(T(x_t))_{t\in\ct}$ is topologically equivalent to the basis of James
tree. Then $T^*$ has non-separable range.
\end{lem}
\begin{proof}
It is essentially a consequence of the Baire Category Theorem.
Indeed assume, towards a contradiction, that there exists a sequence
$(y^*_n)$ in $Y^*$ such that the set $\{T^*(y^*_n):n\in\nn\}$ is norm-dense
in $T^*(Y^*)$. For every $i,m,k\in\nn$ we define
\[ F_{i,m,k}=\big\{ \sg\in 2^\nn: y^*_i\big(T(x_{\sg|n})\big)\geq 2^{-m}
\text{ for every } n\geq k\big\}. \]
Clearly the set $F_{i,m,k}$ is closed. Using the fact that for every $\sg\in 2^\nn$
the sequence $(T(x_{\sg|n}))$ is weak* convergent to an element
$y^{**}_\sg\in Y^{**}\setminus Y$ and our assumption that the sequence
$(x_t)_{t\in\ct}$ is bounded, we see that
\[ 2^\nn =\bigcup_{i,m,k\in\nn} F_{i,m,k}.\]
Therefore, there exist $t_0\in\ct$ and $i_0,m_0,k_0\in\nn$ such that
\[ \{ \sg\in 2^\nn: t_0\sqsubset \sg\}\subseteq F_{i_0,m_0,k_0}.\]
This implies that $y^*_{i_0}\big(T(x_s)\big)\geq 2^{-m_0}$ for every $s\in\ct$
such that $t_0\sqsubset s$ and $|s|\geq k_0$. But the sequence $(T(x_{s_n}))$
is weakly null for every infinite antichain $(s_n)$ in $\ct$, and in particular,
for every infinite antichain $(s_n)$ satisfying $t_0\sqsubset s_n$ and
$|s_n|\geq k_0$ for every $n\in\nn$. Having arrived to the desired
contradiction the proof is completed.
\end{proof}
It remains to show that (1) implies (2). We need to find a sequence
$(x_t)_{t\in\ct}$ in $X$ such that both $(x_t)_{t\in\ct}$ and
$(T(x_t))_{t\in\ct}$ are topologically equivalent to the basis
of James tree. Our strategy is to transform the problem to a
discrete one concerning families of infinite sets. This reduction
will unable us to apply the machinery presented in \S 2.2 and
\S 2.3 and eventually construct the sequence $(x_t)_{t\in\ct}$.

To this end we argue as follows. \textit{We fix a dense sequence $(d_n)$
in the closed unit ball $B_X$ of $X$ and we set $r_n=T(d_n)$ for every
$n\in\nn$}. Notice that the sequence $(d_n)$ is weak* dense in $B_{X^{**}}$.
By the Odell-Rosenthal Theorem \cite{OR} and our assumption that the space
$X$ does not contain a copy of $\ell_1$, we see that $B_{X^{**}}$
consists only of Baire-1 functions. Let $\hhh$ be the weak* closure of the
set $\{r_n:n\in\nn\}$ in $Y^{**}$. Clearly $\hhh$ is weak* compact. Moreover,
we have the following fact.
\begin{fact} \label{f314}
The set $\hhh$ consists only of Baire-1 functions.
\end{fact}
\begin{proof}
By the Main Theorem in \cite{Ro3}, it is enough to show that for every
$N\in [\nn]^{\infty}$ there exists $L=\{l_0<l_1<...\}\in [N]^{\infty}$
such that the sequence $(r_{l_n})$ is weak Cauchy. If this is not the
case, then, by Rosenthal's Dichotomy \cite{Ro2}, there would existed
$M=\{m_0<m_1<...\}\in[\nn]^{\infty}$ such that the sequence $(r_{m_n})$
is equivalent to the standard unit vector basis of $\ell_1$; but then,
the sequence $(d_{m_n})$ would also be equivalent to the standard
unit vector basis of $\ell_1$, a contradiction.
\end{proof}
By the previous remarks, if we deal with a weak* compact subset of
$B_{X^{**}}$ or $\hhh$, then we have at our disposable all classical
machinery for compact subsets of Baire-1 functions discovered in
\cite{BFT,OR,Ro3}.  In what follows, we will use these results without
giving an explicit reference, unless there is some particular need to do so.

We are going to introduce four families of infinite subsets of $\nn$
which are naturally associated to the sequences $(d_n)$ and $(r_n)$.
These families will play a decisive r\^{o}le in the proof.
The first one is defined by
\begin{equation} \label{e32}
\ddd=\{ L\in [\nn]^{\infty}: \text{the sequence } (d_n)_{n\in L}
\text{ is weak* convergent}\}
\end{equation}
while the second one is defined by
\begin{equation} \label{e33}
\rrr=\{ L\in [\nn]^{\infty}: \text{the sequence } (r_n)_{n\in L}
\text{ is weak* convergent}\}.
\end{equation}
Before we give the definition of the next two families, we will
isolate some basic properties of $\ddd$ and $\rrr$.
\begin{fact} \label{f315}
The families $\ddd$ and $\rrr$ are hereditary, co-analytic and cofinal in
$[\nn]^{\infty}$.
\end{fact}
\begin{proof}
It is clear that both $\ddd$ and $\rrr$ are hereditary. It is also
easy to see that they are cofinal in $[\nn]^{\infty}$. To see that
$\ddd$ is co-analytic notice that
\begin{eqnarray*}
L\in \ddd & \Leftrightarrow & \text{the sequence } (d_n)_{n\in L}
\text{ is weak Cauchy} \\
& \Leftrightarrow & \forall x^*\in B_{X^*} \ \forall \ee>0 \
\exists k\in \nn \text{ such that} \\
& & |x^*(d_n)-x^*(d_m)|< \ee \text{ for every } n,m\in L \text{ with }
n,m\geq k.
\end{eqnarray*}
The same argument shows that $\rrr$ is co-analytic. The proof is completed.
\end{proof}
By Fact \ref{f315}, we see that the family $\ddd\cap \rrr$
is hereditary, co-analytic and cofinal in $[\nn]^{\infty}$.
We will need the following stronger property which is
essentially a consequence of the deep effective version of the
Bourgain-Fremlin-Talagrand Theorem due to G. Debs \cite{De1,De2}.
\begin{lem} \label{l316}
There exists a hereditary, Borel and cofinal subfamily $\fff$ of $\ddd\cap\rrr$.
In particular, the family $\fff$ is hereditary, Borel and cofinal in $[\nn]^{\infty}$.
\end{lem}
\begin{proof}
We have already observed that $B_{X^{**}}$ consists only of Baire-1
functions and that the sequence $(d_n)$ is dense in $B_{X^{**}}$.
As it was explained in \cite[Remark 1(2)]{D1},
by Debs' Theorem \cite{De1} (see also \cite{De2}) there exists a hereditary,
Borel and cofinal subfamily $\fff_0$ of $\ddd$. With the same reasoning,
we see that there exists a hereditary, Borel and cofinal subfamily
$\fff_1$ of $\rrr$. We set $\fff=\fff_0\cap\fff_1$. Clearly the
family $\fff$ is as desired. The proof is completed.
\end{proof}
We proceed to define the next two families we mentioned before.
The third one is defined by
\begin{equation} \label{e34}
\ddd_0=\{ L\in [\nn]^{\infty}: \text{the sequence } (d_n)_{n\in L}
\text{ is weakly null}\}.
\end{equation}
Finally, we set
\begin{equation} \label{e35}
\rrr_0=\{ L\in [\nn]^{\infty}: \text{the sequence } (r_n)_{n\in L}
\text{ is weakly null}\}.
\end{equation}
We isolate, below, some structural properties of $\ddd_0$ and $\rrr_0$.
\begin{lem} \label{l317}
Both $\ddd_0$ and $\rrr_0$ are hereditary, co-analytic and M-families.
Moreover, we have $\ddd_0\subseteq \rrr_0$.
\end{lem}
\begin{proof}
It is clear that $\ddd_0\subseteq \rrr_0$ and that $\ddd_0$ and $\rrr_0$
are hereditary. Arguing as in the proof of Fact \ref{f315}, it is easy
to see that they are co-analytic. It remains to check that they are
M-families. We will argue only for the family $\ddd_0$ (the case
of $\rrr_0$ is similarly treated). By Fact \ref{f26}, it is enough
to show that for every sequence $(A_n)$ in $\ddd_0$ there exists
$A\in\ddd_0$ such that $A\cap A_n\neq\varnothing$ for infinitely
many $n\in\nn$. So, let $(A_n)$ be one. We may assume that
$A_n\cap A_m=\varnothing$ if $n\neq m$. For every $n\in\nn$
let $\{a^n_0<a^n_1<...\}$ be the increasing enumeration
of the set $A_n$ and set $x^n_k=d_{a^n_k}$ for every $k\in\nn$.
Since $A_n\in\ddd_0$ the sequence $(x^n_k)$ is weakly null. By Theorem
\ref{t212}, there exists a sequence $(n_i,k_i)$ in $\nn\times\nn$
with $n_i<k_i<n_{i+1}$ and such that the sequence $(x^{n_i}_{k_i})$
is also weakly null. We set $A=\{a^{n_i}_{k_i}:i\in\nn\}$.
Then $A\in\ddd_0$ and $A_{n_i}\cap A\neq\varnothing$ for every
$i\in\nn$. The proof is completed.
\end{proof}
We are about to introduce one more family of infinite subsets of $\nn$.
Let $\fff$ be the family obtained by Lemma \ref{l316}. We set
\begin{equation} \label{e36}
\aaa=\fff\setminus \rrr_0.
\end{equation}
The following lemma is the main technical step towards the proof
of Theorem \ref{t13}.
\begin{lem} \label{l318}
There exists a perfect Lusin gap inside $(\aaa,\ddd_0)$.
\end{lem}
\begin{proof}
It is clear that $\aaa$ and $\ddd_0$ are hereditary and orthogonal.
By Lemma \ref{l316} and Lemma \ref{l317}, we see that $\aaa$ is analytic
and $\ddd_0$ is co-analytic and M-family. By Theorem \ref{t28},
the proof will be completed once we show that $\aaa$ is not countably
generated in $\ddd_0^\perp$. To show this we will argue by contradiction.

So, assume that there exists a sequence $(M_k)$ in $\ddd_0^\perp$ such
that for every $L\in\aaa$ there exists $k\in\nn$ with $L\subseteq^* M_k$.
For every $k\in\nn$ let $\kkk_k$ be the weak* closure of the set
$\{d_n:n\in M_k\}$ in $X^{**}$.
\begin{claim} \label{c319}
For every $k\in\nn$ there exist $F_k\subseteq X^*$ finite and $\ee_k>0$ such that
\[ \kkk_k\cap W(0,F_k,\ee_k)=\varnothing \]
where $W(0,F_k,\ee_k)=\{ x^{**}\in X^{**}: |x^{**}(x^*)|<\ee_k
\text{ for every } x^*\in F_k\}$.
\end{claim}
\begin{proof}[Proof of Claim \ref{c319}]
Fix $k\in\nn$. It is enough to show that $0\notin \kkk_k$.
To see this assume, towards a contradiction, that $0\in\kkk_k$.
Since $\kkk_k\subseteq B_{X^{**}}$ there exists $N\in [M_k]^{\infty}$
such that $N\in\ddd_0$. This contradicts the fact that $M_k\in\ddd_0^\perp$.
The proof of Claim \ref{c319} is completed.
\end{proof}
Let $Z$ be the norm closure of the linear span of the set
\[ F=\bigcup_k F_k. \]
Clearly $Z$ is a norm-separable subspace of $X^*$.
\begin{claim} \label{c320}
We have $T^*(Y^*)\subseteq Z$.
\end{claim}
Granting Claim \ref{c320}, the proof of Lemma \ref{l318} is completed.
Indeed, the inclusion $T^*(Y^*)\subseteq Z$ and the norm-separability
of $Z$ yield that $T^*$ has separable range. This contradicts our
assumption on the operator $T$.

It remains to prove Claim \ref{c320}. Again we will argue
by contradiction. So, assume that $T^*(Y^*)\nsubseteq Z$. There
exist $y^*\in Y^*$, $x^{**}\in X^{**}$ and $\delta>0$ such that
\begin{enumerate}
\item[(a)] $\|T^*(y^*)\|=\|x^{**}\|=1$,
\item[(b)] $Z\subseteq \mathrm{Ker}(x^{**})$ and
\item[(c)] $x^{**}\big(T^*(y^*)\big)>\delta$.
\end{enumerate}
By (a) above, we may select $L\in[\nn]^{\infty}$ such that the sequence
$(d_n)_{n\in L}$ is weak* convergent to $x^{**}$.
By (c), we may assume that $y^*\big(T(d_n)\big)=y^*(r_n)>\delta$
for every $n\in L$, and so, $[L]^{\infty}\cap \rrr_0=\varnothing$.
By Lemma \ref{l316}, the family $\fff$ is hereditary and cofinal
in $[\nn]^{\infty}$. Hence, there exists $A\in[L]^{\infty}$ such
that $[A]^{\infty}\subseteq \aaa$. Recall that the sequence
$(M_k)$ generates $\aaa$. Therefore, there exists $k_0\in \nn$
such that $A\subseteq^* M_{k_0}$. We select $N\in [A]^{\infty}$
with $N\subseteq M_{k_0}$. By Claim \ref{c319}, the set $F_{k_0}$
is finite and $d_n\notin W(0,F_{k_0},\ee_{k_0})$ for every $n\in N$.
Hence, there exist $x_0^*\in F_{k_0}$ and $M\in [N]^{\infty}$ such that
$|x_0^*(d_n)|\geq \ee_{k_0}$ for every $n\in M$. Since the sequence
$(d_n)_{n\in L}$ is weak* convergent to $x^{**}$ and $M\in [L]^{\infty}$
we get that $|x^{**}(x^*_0)|\geq\ee_{k_0}$. In particular,
$x^*_0\notin \mathrm{Ker}(x^{**})$ and $x^*_0\in F_{k_0}\subseteq F\subseteq Z$.
This contradicts property (b) above. Having arrived to the desired
contradiction, the proof of Claim \ref{c320} is completed, and as
we have already indicated, the proof of Lemma \ref{l318}
is also completed.
\end{proof}
We fix a perfect Lusin gap $2^\nn\ni \sg\mapsto (A_\sg,B_\sg)\in \aaa\times \ddd_0$
whose existence is guaranteed by Lemma \ref{l318}. We recall the following
properties of this assignment.
\begin{enumerate}
\item[(P1)] The map $2^\nn\ni\sg\mapsto (A_\sg,B_\sg)\in [\nn]^{\infty}\times
[\nn]^{\infty}$ is one-to-one and continuous.
\item[(P2)] For every $\sg\in 2^\nn$ we have $A_\sg\cap B_\sg=\varnothing$.
\item[(P3)] For every $\sg, \tau\in 2^\nn$ with $\sg\neq\tau$ we have
$(A_\sg\cap B_\tau)\cup (A_\tau\cap B_\sg)\neq\varnothing$.
\end{enumerate}
Let $\sg\in 2^\nn$ be arbitrary. Since $A_\sg\in \aaa\subseteq
(\ddd\cap\rrr)\setminus \rrr_0$ and $\ddd_0\subseteq\rrr_0$ we see
that there exist two non-zero vectors $x^{**}_\sg\in X^{**}$ and
$y^{**}_\sg\in Y^{**}$ such that
\begin{equation} \label{e37}
x^{**}_\sg=\mathrm{weak}^*-\lim_{n\in A_\sg} d_n \ \text{ and } \
y^{**}_\sg=\mathrm{weak}^*-\lim_{n\in A_\sg} r_n.
\end{equation}
Notice that
\begin{equation} \label{e38}
y^{**}_\sg=T^{**}(x^{**}_\sg).
\end{equation}
The following lemma is a consequence of properties (P2) and (P3)
and it is a typical application of related combinatorics (see, for instance,
\cite[Lemma 3.2]{Fa} and the references therein).
\begin{lem} \label{l321}
For every uncountable subset $U$ of $2^\nn$ there exists a sequence
$(\sg_n)$ in $U$ such that the sequences $(x^{**}_{\sg_n})$ and
$(y^{**}_{\sg_n})$ are both weak* convergent to $0$.
\end{lem}
\begin{proof}
By (\ref{e38}) and the weak* continuity of the operator $T^{**}$,
it is enough to find a sequence $(\sg_n)$ in $U$ such that the sequence
$(x^{**}_{\sg_n})$ is weak* convergent to $0$. To this end,
it suffices to show that $0$ belongs to the weak* closure of the set
$\{x^{**}_\sg:\sg\in U\}$ in $X^{**}$. Assume, towards a contradiction,
that this is not the case. It is then possible to select a weak* open
subset $\www$ of $X^{**}$ and a weak* closed subset $\fff$ of $X^{**}$
such that $0\in\www\subseteq\fff$ and $x^{**}_\sg\notin\fff$
for every $\sg\in U$. We set
\begin{equation} \label{e39}
A=\{n\in\nn: d_n\notin\fff\} \ \text{ and } \ B=\{n\in\nn: d_n\in\www\}
\end{equation}
and we notice $A\cap B=\varnothing$. Let $\sg\in U$ be arbitrary.
By (\ref{e37}) and the fact that $x^{**}_\sg\notin\fff$, we see that
$A_\sg\subseteq^* A$. Moreover, since $B_\sg\in\ddd_0$ and $0\in\www$
we have $B_\sg\subseteq^* B$. Therefore, it is possible
to find $k\in\nn$ and an uncountable subset $U'$ of $U$ such that
\begin{equation} \label{e310}
A_\sg\setminus \{0,...,k\}\subseteq A \ \text{ and } \
B_\sg\setminus \{0,...,k\}\subseteq B
\end{equation}
for every $\sg\in U'$. There exist two subsets $F$ and $G$ of $\{0,...,k\}$
and an uncountable subset $U''$ of $U'$ such that
\begin{equation} \label{e311}
A_\sg\cap \{0,...,k\}=F \ \text{ and } \ B_\sg\cap \{0,...,k\}=G
\end{equation}
for every $\sg\in U''$. Notice that $F\cap G=\varnothing$; indeed,
by (\ref{e311}) and property (P2), for every $\sg\in U''$ we have
$F\cap G\subseteq A_\sg\cap B_\sg=\varnothing$.

Let $\sg,\tau\in U''$ with $\sg\neq \tau$. By (\ref{e310})
and (\ref{e311}), we see that
\[ (A_\sg\cap B_\tau)\cup (A_\tau\cap B_\sg)\subseteq
(F\cap G) \cup (A\cap B)=\varnothing.\]
This contradicts property (P3). Having arrived to the desired
contradiction, the proof is completed.
\end{proof}
We should point out that properties (P2) and (P3) will not be used
in the rest of the proof. However, heavy use will be made of property
(P1). We proceed with the following lemma.
\begin{lem} \label{l322}
There exists a perfect subset $P$ of $2^\nn$ such that
$x^{**}_\sg\neq x^{**}_\tau$ and $y^{**}_\sg\neq y^{**}_\tau$
for every $\sg,\tau\in P$ with $\sg\neq \tau$.
\end{lem}
\begin{proof}
For every subset $S$ of $2^\nn$ by $[S]^2$ we denote the set of all
unordered pairs of elements of $S$. We set
\[ \xxx=\big\{ \{\sigma,\tau\}\in [2^\nn]^2: x^{**}_\sg\neq x^{**}_\tau\big\}
\ \text{ and } \
\yyy=\big\{ \{\sigma,\tau\}\in [2^\nn]^2: y^{**}_\sg\neq y^{**}_\tau\big\}.\]
The sets $\xxx$ and $\yyy$ are analytic in $[2^\nn]^2$, in the sense that
the sets
\[ \big\{ (\sigma,\tau)\in 2^\nn\times 2^\nn:
\{\sigma,\tau\}\in\xxx\big\} \ \text{ and } \
\big\{ (\sigma,\tau)\in 2^\nn\times 2^\nn:
\{\sigma,\tau\}\in\yyy\big\}\]
are both analytic subsets of $2^\nn\times 2^\nn$. Indeed, by (\ref{e37}), we have
\begin{eqnarray*}
\{\sg,\tau\}\in\xxx & \Leftrightarrow & \exists x^{*}\in B_{X^*} \ \exists \ee>0
\ \exists k\in\nn \text{ such that } |x^*(d_n)-x^*(d_m)|\geq \ee \\
& & \text{ for every } n\in A_\sg \text{ and every } m\in A_\tau \text{ with } n,m\geq k.
\end{eqnarray*}
Since the map $2^\nn\ni\sg\mapsto A_\sg\in [\nn]^{\infty}$ is continuous,
the above equivalence yields that the set $\xxx$ is analytic. With the same
reasoning and using the continuity of the map $2^\nn\ni\sg\mapsto B_\sg\in [\nn]^{\infty}$
we see that $\yyy$ is also analytic. By a result of F. Galvin (see
\cite[Theorem 19.7]{Kechris}), there exists a perfect subset $P$ of $2^\nn$
such that one of the following cases occur.
\medskip

\noindent \textsc{Case 1:} $[P]^2\cap\xxx=\varnothing$. In this case we see that
there exists a non-zero vector $x^{**}\in X^{**}$ such that $x^{**}_\sg=x^{**}$
for every $\sg\in P$. This is impossible by Lemma \ref{l321}.
\medskip

\noindent \textsc{Case 2:} $[P]^2\cap\yyy=\varnothing$. As above, we see that
there exists a non-zero vector $y^{**}\in Y^{**}$ such that $y^{**}_\sg=y^{**}$
for every $\sg\in P$. This is also impossible.
\medskip

\noindent \textsc{Case 3:} $[P]^2\subseteq (\xxx\cap\yyy)$. Notice that, in
this case, we have $x^{**}_\sg\neq x^{**}_\tau$ and $y^{**}_\sg\neq y^{**}_\tau$
for every $\sg,\tau\in P$ with $\sg\neq\tau$. Therefore, the perfect set $P$
is as desired. The proof is completed.
\end{proof}
So far we have been working with the perfect Lusin gap inside $(\aaa,\ddd_0)$.
The next lemma we will unable us to start the process for selecting the sequence
$(x_t)_{t\in\ct}$.
\begin{lem} \label{l323}
Let $P$ be the perfect subset of $2^\nn$ obtained by Lemma \ref{l322}. Then
there exist a sequence $(k_t)_{t\in\ct}$ in $\nn$ and a continuous, one-to-one
map $h:2^\nn\to P$ with the following properties.
\begin{enumerate}
\item[(i)] For every $t,t'\in\ct$ with $t\neq t'$ we have $k_t\neq k_{t'}$.
\item[(ii)] For every $\tau\in 2^\nn$ we have $\{k_{\tau|n}:n\in\nn\}\subseteq A_{h(\tau)}$.
\end{enumerate}
\end{lem}
\begin{proof}
Every infinite subset of $\nn$ is naturally identified with an element of $2^\nn$.
Therefore, the map $P\ni \sg\mapsto A_\sg\in 2^\nn$ is continuous and one-to-one.
Let $F$ be its image and denote by $f:F\to P$ its inverse. Notice that $F$ is
closed and that $f$ is a homeomorphism. There exists a downwards closed subtree
$R$ of $\ct$ such that $[R]=F$. Observe that $R$ is a perfect subtree; that is,
every $t\in R$ has two incomparable extensions in $R$. Hence, it is possible
to select a dyadic subtree $D=\{s_t:t\in\ct\}$ of $\ct$ such that $D\subseteq R$
and with the following properties.
\begin{enumerate}
\item[(a)] For every $t\in\ct$ the node $s_t$ ends with $1$; that is, there exists
$w_t\in\ct$ such that $s_t=w_t^{\con}1$.
\item[(c)] For every $t,t'\in\ct$ with $t\neq t'$ we have $|s_t|\neq |s_{t'}|$.
\end{enumerate}
We set $k_t=|s_t|-1$ for every $t\in\ct$ and we define $h:2^\nn\to P$ by the rule
\[ h(\tau)= f\Big( \bigcup_{n\in\nn} s_{\tau|n} \Big).\]
The sequence $(k_t)_{t\in\ct}$ and the map $h$ are as desired.
\end{proof}
Let $(k_t)_{t\in\ct}$ be the sequence in $\nn$ obtained by Lemma \ref{l323}.
For every $t\in\ct$ we define
\begin{equation} \label{e312}
e_t=d_{k_t}
\end{equation}
The desired sequence $(x_t)_{t\in\ct}$ will be a subsequence of $(e_t)_{t\in\ct}$
of the form $(e_{s_t})_{t\in \ct}$ where $\{s_t:t\in\ct\}$ is dyadic subtree
of $\ct$. We isolate, for future use, the following immediate consequence
of Lemma \ref{l323}.
\begin{enumerate}
\item[(P4)] For every $\tau\in 2^\nn$ the sequences $(e_{\tau|n})$ and
$(T(e_{\tau|n}))$ are weak* convergent to the non-zero vectors
$x^{**}_{h(\tau)}$ and $y^{**}_{h(\tau)}$ respectively.
\end{enumerate}
\begin{lem} \label{l324}
There exist a regular dyadic subtree $D_0$ of $\ct$ and a constant $\theta>0$
such that $\|T(e_t)\|\geq \theta$ for every $t\in D_0$.
\end{lem}
\begin{proof}
We will show that there exist $s_0\in\ct$ and $\theta>0$ such that for every
$t\in \ct$ with $s_0\sqsubseteq t$ we have $\|T(e_t)\|\geq \theta$.
In such a case, the regular dyadic subtree $D_0=\{t\in\ct:s_0\sqsubseteq t\}$
and the constant $\theta$ satisfy the requirements of the lemma.

To find the node $s_0$ and the constant $\theta$ we will argue by contradiction.
So, assume that for every $s\in\ct$ and every $\theta>0$ there exists $t\in\ct$
with $s\sqsubseteq t$ and such that $\|T(e_t)\|\leq \theta$. Hence, it is possible
to select a sequence $(t_k)$ in $\ct$ such that for every $k\in\nn$ we have
\begin{enumerate}
\item[(a)] $t_k\sqsubset t_{k+1}$ and
\item[(b)] $\|T(e_{t_k})\|\leq 2^{-k}$.
\end{enumerate}
By (a) above, the set $\{t_k:k\in\nn\}$ is an infinite chain. We set
\[ \tau=\bigcup_{k\in\nn} t_k\in 2^\nn.\]
By property (P4), the sequence $(T(e_{\tau|n}))$ is weak* convergent
to the non-zero vector $y^{**}_{h(\tau)}$. Hence, so is the sequence
$(T(e_{t_k}))$. By (b) above, we see that $y^{**}_{h(\tau)}=0$,
a contradiction. The proof is completed.
\end{proof}
\begin{lem} \label{l325}
There exists a regular dyadic subtree $D_1$ of $\ct$ with $D_1\subseteq D_0$
and such that for every infinite antichain $A$ of $D_1$ the sequence
$(e_t)_{t\in A}$ is weakly null.
\end{lem}
\begin{proof}
Consider the subset $\ccc$ of $D_0^\nn$ defined by
\[ (s_n)\in\ccc\Leftrightarrow \text{the sequence }
(e_{s_n}) \text{ is weakly null}.\]
It is easy to check that $\ccc$ is a co-analytic subset of $D_0^\nn$.
Applying Theorem \ref{t211} for the increasing antichains of $D_0$ and
the color $\ccc$, we find a regular dyadic subtree $R$ of $\ct$ with
$R\subseteq D_0$ and such that either $\incr(R)\subseteq\ccc$ or
$\incr(R)\cap\ccc=\varnothing$. Next, applying Theorem \ref{t211}
for the decreasing antichains of $R$ and the same color, we find
a regular dyadic subtree $D_1$ of $\ct$ with $D_1\subseteq R$ and
such that either $\decr(D_1)\subseteq\ccc$ or $\decr(D_1)\cap\ccc=\varnothing$.
We will show that the regular dyadic subtree $D_1$ is the desired one.
Indeed, notice that $D_1\subseteq D_0$. To show that for every infinite
antichain $A$ of $D_1$ the sequence $(e_t)_{t\in A}$ is weakly null,
we will show first the following weaker property.
\begin{claim} \label{c326}
Either $\incr(D_1)\subseteq\ccc$ or $\decr(D_1)\subseteq\ccc$.
\end{claim}
\begin{proof}[Proof of Claim \ref{c326}]
Let $\kkk$ be the weak* closure of the set $\{e_t:t\in D_1\}$ in $X^{**}$.
By property (P4), we have that $x^{**}_{h(\tau)}\in\kkk$ for every
$\tau\in [\hat{D_1}]$. The map $h$ is one-to-one. Therefore, the set
$U=\{h(\tau):\tau\in [\hat{D_1}]\}$ is uncountable. By Lemma \ref{l321},
there exists a sequence $(\tau_n)$ in $[\hat{D_1}]$ such that the
sequence $(x^{**}_{h(\tau_n)})$ is weak* convergent to $0$. Hence,
$0\in\kkk$. It is then possible to select an infinite subset $S$
of $D_1$ such that the sequence $(e_t)_{t\in S}$ is weakly null.
Applying the classical Ramsey Theorem, we find an infinite subset $S'$
of $S$ which is either a chain or an antichain. Notice that $S'$
has to be an antichain (for is not, there would existed $\tau\in [\hat{D_1}]$
such that $x^{**}_{h(\tau)}=0$). By part (ii) of Lemma \ref{l210},
there exists a sequence $(s_n)$ in $S'$ such that either $(s_n)\in \incr(D_1)$
or $(s_n)\in\decr(D_1)$. If $(s_n)\in\incr(D_1)$, then, by part (iii) of Lemma
\ref{l210}, we see that $\incr(R)\cap\ccc\neq\varnothing$ and so
$\incr(D_1)\subseteq\incr(R)\subseteq\ccc$. Otherwise,
$\decr(D_1)\cap\ccc\neq\varnothing$ which yields that $\decr(D_1)\subseteq\ccc$.
The proof of Claim \ref{c326} is completed.
\end{proof}
Next we strengthen the conclusion of Claim \ref{c326} as follows.
\begin{claim} \label{c327}
We have $\incr(D_1)\subseteq\ccc$ and $\decr(D_1)\subseteq\ccc$.
\end{claim}
\begin{proof}[Proof of Claim \ref{c327}]
By Claim \ref{c326}, either $\incr(D_1)\subseteq\ccc$ or
$\decr(D_1)\subseteq\ccc$. As the argument is symmetric, we will assume that
$\incr(D_1)\subseteq\ccc$. Recursively, for every $n\in\nn$ we select
an infinite antichain $(t^n_k)$ of $D_1$ such that the following are satisfied.
\begin{enumerate}
\item[(a)] For every $n\in\nn$ we have $(t^n_k)\in\incr(D_1)$.
\item[(b)] For every $n,n'\in\nn$ with $n<n'$ and every $k,l\in\nn$
we have $t^{n'}_k\prec t^n_l$.
\end{enumerate}
The recursive selection is fairly standard and the details are left to the reader.
By (a) above and our assumption that $\incr(D_1)\subseteq\ccc$, we see that
for every $n\in\nn$ the sequence $(e_{t^n_k})$ is weakly null. By Theorem
\ref{t212}, there exists a sequence $(n_i,k_i)$ in $\nn\times\nn$
with $n_i<k_i<n_{i+1}$ for every $i\in\nn$ and such that the sequence
$(e_{t^{n_i}_{k_i}})$ is also weakly null. By (b), we see that
\begin{enumerate}
\item[(c)] $t^{n_{i'}}_{k_{i'}}\prec t^{n_i}_{k_i}$ for every $i,i'\in\nn$
with $i<i'$.
\end{enumerate}
By part (ii) of Lemma \ref{l210}, there exists a subsequence of
$(t^{n_i}_{k_i})$, denoted for simplicity by $(s_m)$, such that
either $(s_m)\in\incr(D_1)$ or $(s_m)\in\decr(D_1)$. Invoking (c),
we get that $(s_m)\in\decr(D_1)$. Since the sequence $(e_{s_m})$
is weakly null, we conclude that $\decr(D_1)\cap\ccc\neq\varnothing$
and so $\decr(D_1)\subseteq\ccc$. The proof of Claim \ref{c327}
is completed.
\end{proof}
We are now ready to check that the sequence $(e_t)_{t\in A}$ is
weakly null for every infinite antichain $A$ of $D_1$. So let $A$
be one. Let $B$ be an arbitrary infinite subset of $A$. By part
(ii) of Lemma \ref{l210}, there exists an infinite sequence
$(s_n)$ in $B$ such that either $(s_n)\in\incr(D_1)$ or
$(s_n)\in\decr(D_1)$. By Claim \ref{c327}, we see that the sequence
$(e_{s_n})$ is weakly null. In other words, every subsequence of
$(e_t)_{t\in A}$ has a further weakly null subsequence. This
yields that the entire sequence $(e_t)_{t\in A}$ is weakly
null. Thus, the proof of Lemma \ref{l325} is completed.
\end{proof}
As we have already mentioned in the introduction, by $\varphi:\ct\to\nn$
we denote the unique bijection satisfying $\varphi(t)<\varphi(t')$
if either $|t|<|t'|$ or $|t|=|t'|$ and $t<_{\mathrm{lex}}t'$.
\begin{lem} \label{l328}
There exists a dyadic subtree
$D_2=\{s_t:t\in\ct\}$ of $\ct$ such that $D_2\subseteq D_1$
and with the following property. If $(t_n)$ is the enumeration
of $\ct$ according to the bijection $\varphi$, then the sequences
$(e_{s_{t_n}})$ and $(T(e_{s_{t_n}}))$ are both seminormalized
basic sequences.
\end{lem}
\begin{proof}
Notice, first, that the sequences $(e_t)_{t\in D_1}$ and $(T(e_t))_{t\in D_1}$
are seminormalized. Indeed, $D_1\subseteq D_0$ and so, by Lemma \ref{l324},
for every $t\in D_1$ we have
\[ \theta\leq \|T(e_t)\|\leq \|T\| \ \text{ and } \ \theta\cdot \|T\|^{-1}
\leq \|e_t\|\leq 1.\]
Let $t\in D_1$ be arbitrary. We select an infinite antichain $A$ of $D_1$
such that $t\sqsubset s$ for every $s\in A$. By Lemma \ref{l325}, the
sequences $(x_s)_{s\in A}$ and $(T(x_s))_{s\in A}$ are both weakly null.
Using this observation and the classical procedure of Mazur for selecting
basic sequences (see \cite{LT}), the result follows.
\end{proof}
Let $D_2=\{s_t:t\in\ct\}$ be the dyadic subtree obtained by Lemma
\ref{l328}. For every $t\in\ct$ we define
\begin{equation} \label{e313}
x_t=e_{s_t}
\end{equation}
We will show that the sequence $(x_t)_{t\in\ct}$ is the desired one.
\medskip

\noindent \textbf{(1)} Let $(t_n)$ be the enumeration of $\ct$ according
to the bijection $\varphi$. By Lemma \ref{l328}, we have that $(x_{t_n})$
and $(T(x_{t_n}))$ are seminormalized basic sequences.
\medskip

\noindent \textbf{(2)} Let $A$ be an infinite antichain of $\ct$. Notice
that the set $A'=\{s_t:t\in A\}$ is an infinite antichain of $D_2$. Since
$D_2\subseteq D_1$, by Lemma \ref{l325}, we see that the sequences
$(x_t)_{t\in A}$ and $(T(x_t))_{t\in A}$ are both weakly null.
\medskip

\noindent \textbf{(3)} Let $\sg\in 2^\nn$ be arbitrary. Observe that
the set $\{s_{\sg|n}:n\in\nn\}$ is an infinite chain of $\ct$. We define
\[ \tau_\sg=\bigcup_{n\in\nn} s_{\sg|n}\in 2^\nn\]
and we notice that $(x_{\sg|n})$ is a subsequence of $(e_{\tau_\sg|n})$.
By property (P4), we see that the sequences $(x_{\sg|n})$ and
$(T(x_{\sg|n}))$ are weak* convergent to the non-zero vectors
$x^{**}_{h(\tau_{\sg})}$ and $y^{**}_{h(\tau_{\sg})}$ respectively.

Next we check that $x^{**}_{h(\tau_{\sg})}\in X^{**}\setminus X$.
Assume on the contrary that $x^{**}_{h(\tau_{\sg})}\in X$. Let
$(t_n)$ be the enumeration on $\ct$ according to the bijection
$\varphi$ and observe that $(x_{\sg|n})$ is a subsequence of
$(x_{t_n})$. By Lemma \ref{l328}, we get that $(x_{\sg|n})$ is
a basic sequence which is weakly convergent to
$x^{**}_{h(\tau_{\sg})}\in X$. This implies that
$x^{**}_{h(\tau_{\sg})}=0$, a contradiction.
Therefore, $x^{**}_{h(\tau_{\sg})}\in X^{**}\setminus X$.
With the same reasoning we verify that
$y^{**}_{h(\tau_{\sg})}\in Y^{**}\setminus Y$.

Finally, let $\sg,\sg'\in 2^\nn$ with $\sg\neq\sg'$. Notice
that $\tau_{\sg}\neq\tau_{\sg'}$. The map $h$ obtained by
Lemma \ref{l323} is one-to-one. Therefore, $h(\tau_\sg)\neq
h(\tau_{\sg'})$. By Lemma \ref{l322}, we conclude that
$x^{**}_{h(\tau_{\sg})}\neq x^{**}_{h(\tau_{\sg'})}$
and $y^{**}_{h(\tau_{\sg})}\neq y^{**}_{h(\tau_{\sg'})}$.
\medskip

Having verified that the sequences $(x_t)_{t\in\ct}$ and
$(T(x_t))_{t\in\ct}$ are both topologically equivalent
to the basis of James tree, the proof of Theorem \ref{t13} is completed.


\section{Proof of Theorem \ref{t14}}

This section is devoted to the proof of Theorem \ref{t14}. Let us
first argue that (2) implies (1). If the operator $T$ fixes a copy
of $\ell_1$, then clearly $T^*$ has non-separable range. If, alternatively,
there exists a bounded sequence $(x_t)_{t\in\ct}$ in $X$ such that
its image $(T(x_t))_{t\in\ct}$ is topologically equivalent
to the basis of James tree, then the non-separability of the
range of $T^*$ is guaranteed by Lemma \ref{l313}.

We work now to prove that (1) implies (2). As in the proof of
Theorem \ref{t13}, we fix a dense sequence $(d_n)$ in $B_X$
and we set $r_n=T(d_n)$ for every $n\in\nn$. We distinguish
the following cases.
\medskip

\noindent \textsc{Case 1:} \textit{There exists a subsequence
$(r_{l_n})$ of $(r_n)$ which is equivalent to the standard unit
vector basis of $\ell_1$.} Let $E$ be the closed subspace of $X$
spanned by the corresponding subsequence $(d_{l_n})$ of $(d_n)$.
Notice that $E$ is isomorphic to $\ell_1$ and that $T:E\to Y$
is an isomorphic embedding. Hence, in this case we see that
the operator $T$ fixes a copy of $\ell_1$.
\medskip

\noindent \textsc{Case 2:} \textit{No subsequence of $(r_n)$ is
equivalent to the standard unit vector basis of $\ell_1$.} We will
show that there exists a bounded sequence $(x_t)_{t\in\ct}$ in $X$
such that its image $(T(x_t))_{t\in\ct}$ is topologically equivalent
to the basis of James tree. The proof is similar to the proof of
Theorem \ref{t13}, and so, we shall only indicate the necessary
changes.

First we observe that, by Rosenthal's Dichotomy \cite{Ro2} and
our assumption, every subsequence of $(r_n)$ has a further weak
Cauchy subsequence. Therefore, by the Main Theorem in \cite{Ro3},
the weak* closure $\hhh$ of the set $\{r_n:n\in\nn\}$ in $Y^{**}$
consists only of Baire-1 functions. We define the families
$\rrr$ and $\rrr_0$ exactly as we did in (\ref{e33}) and
(\ref{e35}) respectively. We recall that both $\rrr$ and
$\rrr_0$ are hereditary and co-analytic. Moreover, $\rrr$
is cofinal in $[\nn]^{\infty}$ while $\rrr_0$ is an M-family.
Arguing as in the proof of Lemma \ref{l316}, we see that
there exists a hereditary, Borel and cofinal subfamily
$\fff'$ of $\rrr$. We define
\begin{equation} \label{e414}
\aaa'=\fff'\setminus \rrr_0.
\end{equation}
We have the following analogue of Lemma \ref{l318}.
\begin{lem} \label{l429}
There exists a perfect Lusin gap inside $(\aaa',\rrr_0)$.
\end{lem}
Granting Lemma \ref{l429}, the rest of the proof of Theorem
\ref{t14} is the same to that of Theorem \ref{t13} mutatis mutandis.

So, what remains is to prove Lemma \ref{l429}. By Theorem \ref{t28},
it is enough to show that the family $\aaa'$ is not countably generated
in the family $\rrr_0^\perp$. If this is not the case, then
there exists a sequence $(N_k)$ in $\rrr_0^\perp$ such that
for every $L\in \aaa'$ there exists $k\in\nn$ with
$L\subseteq^* N_k$. For every $k\in\nn$ let $\hhh_k$
be the weak* closure of the set $\{r_n:n\in N_k\}$
in $Y^{**}$. The fact that $N_k\in \rrr_0^\perp$
reduces to the fact that $0\notin \hhh_k$. Therefore,
there exist $E_k\subseteq Y^*$ finite and $\ee_k>0$ such
that $\hhh_k\cap W(0,E_k,\ee_k)=\varnothing$. Let $E$
be the norm closure of the linear span of the set
\[ E=\bigcup_k E_k. \]
The proof will be completed once we show that $T^*(E)$ is
norm dense in $T^*(Y^*)$. To this end, we will argue
by contradiction. So, assume that there exist $x^{**}\in X^{**}$,
$y^*\in Y^*$ and $\delta>0$ such that
\begin{enumerate}
\item[(a)] $\|x^{**}\|=\|T^*(y^*)\|=1$,
\item[(b)] $T^*(E)\subseteq \mathrm{Ker}(x^{**})$ and
\item[(c)] $x^{**}\big(T^*(y^*)\big)>\delta$.
\end{enumerate}
Since $T^{**}(B_{X^{**}})\subseteq \hhh$ and $\hhh$ consists only
of Baire-1 functions we may select $L\in\rrr$ such that the
sequence $(r_n)_{n\in L}$ is weak* convergent to $T^{**}(x^{**})$.
By (c) above, we may assume that $y^*(r_n)>\delta$ for every
$n\in L$, and so, $[L]^{\infty}\cap \rrr_0=\varnothing$.
Since the family $\fff'$ is cofinal in $[\nn]^{\infty}$ and the
sequence $(N_k)$ generates $\aaa'$, it is possible to select
$k_0\in \nn$, $y_0^*\in E_{k_0}$ and $A\in [L]^{\infty}$ such
that $|y^*_0(r_n)|\geq \ee_{k_0}$ for every $n\in A$. This
implies that $T^*(y^*_0)\notin \mathrm{Ker}(x^{**})$ which
contradicts property (b) above. Having arrived to the desired
contradiction the proof of Lemma \ref{l429} is completed,
and as we have already indicated, the proof of Theorem
\ref{t14} is also completed.


\section{Comments}

\subsection{ }

Theorem \ref{t13} and Theorem \ref{t14} were motivated by the
structural results in \cite{ADK1,ADK2} and our recent work on
quotients of separable Banach spaces in \cite{D2} where a special
case of Theorem \ref{t13} was proved and applied. Results of this
type are, typically, used to reduce the existence of an
\textit{uncountable} family to the existence of a canonical
\textit{countable} object which is much more amenable to
combinatorial manipulations.

\subsection{ }

We have already mentioned in the introduction that if an operator
$T:X\to Y$ fixes a copy of a sequence $(x_t)_{t\in\ct}$ topologically
equivalent to the basis of James tree, then the topological
structure of the weak* closure of $\{x_t:t\in\ct\}$ in $X^{**}$
is preserved under the action on $T^{**}$. Precisely, we have the following.
\begin{lem} \label{l530}
Let $X$ and $Y$ be Banach spaces and $T:X\to Y$ be an operator.
Assume that there exists a sequence $(x_t)_{t\in\ct}$ in $X$ such that
both $(x_t)_{t\in\ct}$ and $(T(x_t))_{t\in\ct}$ are topologically
equivalent to the basis of James tree. Let $\kkk$ and $\hhh$ be the
weak* closures of $\{x_t:t\in\ct\}$ and $\{T(x_t):t\in\ct\}$
respectively. Then $T^{**}:\kkk\to\hhh$ is a weak* homeomorphism.
\end{lem}
\begin{proof}
Clearly we may assume that $X$ and $Y$ are separable.
We observe that $\kkk$ and $\hhh$ consist only of Baire-1 functions.
\begin{claim} \label{c531}
The weak* isolated points of $\kkk$ is the set $\{x_t:t\in\ct\}$.
Respectively, the weak* isolated points of $\hhh$ is the set $\{T(x_t):t\in\ct\}$.
\end{claim}
\begin{proof}[Proof of Claim \ref{c531}]
We will argue only for the set $\kkk$ (the argument for the set $\hhh$ is
identical). Let $\iii$ be the set of all weak* isolated points of $\kkk$.
Let $x^{**}\notin\iii$ be arbitrary and select an infinite subset $A$ of
$\ct$ such that the sequence $(x_t)_{t\in A}$ is weak* convergent to $x^{**}$.
If $A$ contains an infinite antichain, then $x^{**}=0$. Otherwise, there
exists $\sg\in 2^\nn$ such that $x^{**}=x^{**}_\sg\in X^{**}\setminus X$.
It follows that $\iii\subseteq\{x_t:t\in\ct\}$. To see the other inclusion
assume, towards a contradiction, that there exists $s\in\ct$ such that
$x_s\notin\iii$. Let $(t_n)$ be the enumeration of $\ct$ according to the bijection
$\varphi$. There exists a subsequence $(x_k)$ of $(x_{t_n})$ which
is weakly convergent to $x_s$. Since $(x_k)$ is basic we get that
$x_s=0$, a contradiction. The proof of Claim \ref{c531} is completed.
\end{proof}
\begin{claim} \label{c532}
For every infinite subset $S$ of $\ct$ the sequence $(x_t)_{t\in S}$
is weak* convergent if and only if the sequence $(T(x_t))_{t\in S}$
is weak* convergent.
\end{claim}
\begin{proof}[Proof of Claim \ref{c532}]
Let $E$ be a Banach space and $(e_t)_{t\in\ct}$ be a
sequence in $E$ which is topologically equivalent to the basis
of James tree. Let $S$ be an arbitrary subset of $\ct$. By Definition
\ref{d11}, we see that the sequence $(e_t)_{t\in S}$ is weak*
convergent if and only if either
\begin{enumerate}
\item[(a)] there exists an infinite antichain $A$ of $\ct$
such that $S\subseteq^* A$ or
\item[(b)] there exists $\sg\in 2^\nn$ such that
$S\subseteq^* \{\sg|n:n\in\nn\}$.
\end{enumerate}
Using this observation, the result follows.
\end{proof}
By Claim \ref{c531}, Claim \ref{c532} and the fact that $\kkk$ and $\hhh$
consist only of Baire-functions, we may apply Lemma 19 in \cite{ADK1}
to infer that the map
\[ \kkk\ni x_t\mapsto T(x_t)\in\hhh\]
is extended to a weak* homeomorphism $\Phi:\kkk\to\hhh$. Using the
weak* continuity of $T^{**}$ we see that $T^{**}|_\kkk=\Phi$. The
proof of Lemma \ref{l530} is completed.
\end{proof}

\subsection{ }

Recall that a non-empty finite subset $\seg$ of $\ct$ is said to be
a \textit{segment} if there exist $s,t\in\ct$ with $s\sqsubseteq t$
and such that $\seg=\{w\in\ct: s\sqsubseteq w\sqsubseteq t\}$.
Let $1<p<+\infty$. The \textit{$p$-James tree space}, denoted
by $JT_p$, is defined to be the completion of $c_{00}(\ct)$
under the norm
\[ \|x\|_{JT_p}=\sup \Big\{ \Big( \sum_{i=0}^d \big|
\sum_{t\in \seg_i} x(t)\big|^p \Big)^{1/p} \Big\} \]
where the above supremum is taken over all families
$(\seg_i)_{i=0}^d$ of pairwise disjoint segments of $\ct$.
The classical James tree space is the space $JT_2$.

Let $(e^p_t)_{t\in\ct}$ be the standard Hamel basis of
$c_{00}(\ct)$ regarded as a sequence in $JT_p$. If $(t_n)$
is the enumeration of $\ct$ according to the bijection
$\varphi$, then the sequence $(e^p_{t_n})$ defines
a normalized Schauder basis of $JT_p$. It is easy to
check that $(e^p_t)_{t\in\ct}$ is topologically equivalent
to the basis of James tree according to Definition \ref{d11}
(a fact that actually justifies our terminology).

It is well-known that the space $JT_p$ is hereditarily $\ell_p$;
that is, every infinite-dimensional subspace of $JT_p$ contains a copy
of $\ell_p$ (see \cite{Ja}). In particular, if $1<p<q<+\infty$,
then every operator $T:JT_p\to JT_q$ is strictly singular.
Nevertheless, there do exist operators in $\mathcal{L}(JT_p,JT_q)$
fixing a copy of a sequence topologically equivalent to
the basis of James tree. The natural inclusion map
$I_{p,q}:JT_p\to JT_q$ is an example.


\end{document}